\documentclass[12pt]{amsart}
\usepackage{amsfonts}
\usepackage{graphicx}
\usepackage{amsmath}
\usepackage{amssymb}
\usepackage{amsthm}
\usepackage{enumerate}
\usepackage[all, cmtip]{xy} 
\newtheorem{theorem}{Theorem}

\newtheorem{corollary}[theorem]{Corollary}

\newtheorem{lemma}[theorem]{Lemma}

\newcommand{\End}{\operatorname{End}}
\begin{document}
\title[Extension of a result of Dickson and Fuller]{Additive Unit Representations in Endomorphism Rings and an Extension of a result of Dickson and Fuller}
\author{Pedro A. Guil Asensio}
\address{Departamento de Mathematicas, Universidad de Murcia, Murcia, 30100, Spain}
\email{paguil@um.es}
\author{Ashish K. Srivastava}
\address{Department of Mathematics and Computer Science, St. Louis University, St.
Louis, MO-63103, USA}
\email{asrivas3@slu.edu}
\keywords{automorphism-invariant modules, injective modules, quasi-injective modules}
\subjclass[2000]{16D50, 16U60, 16W20}
\dedicatory{Dedicated to T. Y. Lam on his 70th Birthday}

\begin{abstract}
A module is called automorphism-invariant if it is invariant under any automorphism of its injective hull. Dickson and Fuller have shown that if $R$ is a finite-dimensional algebra over a field $\mathbb F$ with more than two elements then an indecomposable automorphism-invariant right $R$-module must be quasi-injective. In this note, we extend and simplify the proof of this result by showing that any automorphism-invariant module over an algebra over a field with more than two elements is quasi-injective. Our proof is based on the study of the additive unit structure of endomorphism rings.
\end{abstract}

\maketitle

\section{Introduction.}

\noindent The study of the additive unit structure of rings has a long tradition. The earliest instance may be found in the investigations of Dieudonn\'{e} on Galois theory of
simple and semisimple rings \cite{D}. In \cite{H}, Hochschild studied additive unit representations of elements in simple algebras and proved that each element of a simple algebra over any field is a sum of units. Later, Zelinsky \cite{Zelinsky} proved that every linear transformation of a vector space $V$ over a division ring $D$ is the sum of two invertible linear transformations except when $V$ is one-dimensional over $\mathbb F_2$. Zelinsky also noted in his paper that this result follows from a previous result of Wolfson \cite{Wolfson}.

The above mentioned result of Zelinsky has been recently extended by Khurana and Srivastava in \cite{KS} where they proved that any element in the endomorphism ring of a continuous module $M$ is a sum of two automorphisms if and only if $\End(M)$ has no factor ring isomorphic to the field of two elements $\mathbb F_2$. In particular, this means that, in order to check if a module $M$ is invariant under endomorphisms of its injective hull $E(M)$, it is enough to check it under automorphisms, provided that $\End(E(M))$ has no factor ring isomorphic to $\mathbb F_2$. Recall that a module $M$ is called  \textit{quasi-injective} if every homomorphism from a submodule $L$ of $M$ to $M$ can be extended to an endomorphism of  $M$. Johnson and Wong characterized quasi-injective modules as those that are invariant under any endomorphism of their injective hulls \cite{JW}. 

A module $M$ which is invariant under automorphisms of its injective hull is called an {\it automorphism-invariant module}. This class of modules was first studied by Dickson and Fuller in \cite{DF} for the particular case of finite-dimensional algebras over fields $\mathbb F$ with more than two elements. They proved  that if $R$ is a finite-dimensional algebra over a field $\mathbb F$ with more than two elements then an indecomposable automorphism-invariant right $R$-module must be quasi-injective. And it has been recently shown in \cite{SS} that this result fails to hold if $\mathbb F$ is a field of two elements. Let us recall that a ring $R$ is said to be of {\it right invariant module type} if every indecomposable right $R$-module is quasi-injective. Thus, the result of Dickson and Fuller states that if $R$ is a finite-dimensional algebra over a field $\mathbb F$ with more than two elements, then $R$ is of right invariant module type if and only if every indecomposable right $R$-module is automorphism-invariant. Examples of automorphism-invariant modules which are not quasi-injective, can be found in \cite{ESS} and \cite{Teply}. And recently, it has been shown in \cite{ESS} that a module $M$ is automorphism-invariant if and only if every monomorphism from a submodule of $M$ extends to an endomorphism of $M$. For more details on automorphism-invariant modules, see \cite{ESS}, \cite{LZ}, \cite{SS}, and \cite{AS}. 

The purpose of this note is to exploit the above mentioned result of Khurana and Srivastava in \cite{KS} in order to extend, as well as to give a much easier proof, of Dickson and Fuller's result by showing that if $M$ is any right $R$-module such that there are no ring homomorphisms from $\End_R(M)$ into the field of two elements $\mathbb{F}_2$, then $M_R$ is automorphism-invariant if and only if it is quasi-injective. In particular, we deduce that if $R$ is an algebra over a field $\mathbb F$ with more than two elements, then a right $R$-module $M$ is automorphism-invariant if and only if it is quasi-injective.

%Along this paper, all rings will be associative rings with identity and module will mean a unitary right module. 
Throughout this paper, $R$ will always denote an associative ring with identity element and modules will be right unital. We refer to \cite{AF} for any undefined notion arising in the text.

\section*{Results.}

We begin this section by proving a couple of lemmas that we will need in our main result.

\begin{lemma} \label{basic}
Let $M$ be a right $R$-module such that $\End(M)$ has no factor isomorphic to $\mathbb{F}_2$. Then $\End(E(M))$ has no factor isomorphic to $\mathbb{F}_2$ either. 
\end{lemma} 

\begin{proof}
Let $M$ be any right $R$-module such that $\End(M)$ has no factor isomorphic to $\mathbb{F}_2$ and let $S=\End(E(M))$. We want to show that $S$ has no factor isomorphic to $\mathbb{F}_2$. Assume to the contrary that $\psi: S\rightarrow \mathbb{F}_2$ is a ring homomorphism. As $\mathbb{F}_2\cong\End_{\mathbb Z}(\mathbb{F}_2)$, the above ring homomorphism yields a right $S$-module structure to $\mathbb{F}_2$. Under this right $S$-module structure, $\psi:S\rightarrow \mathbb{F}_2$ becomes a homomorphism of $S$-modules. Moreover, as $\mathbb{F}_2$ is simple as $\mathbb Z$-module, so is as right $S$-module. Therefore, $\ker(\psi)$ contains the Jacobson radical $J(S)$ of $S$ and thus, it factors through a ring homomorphism $\psi':S/J(S)\rightarrow \mathbb{F}_2$.

On the other hand, given any endomorphism $f:M\rightarrow M$, it extends by injectivity to a (non-unique) endomorphism $\varphi_f: E(M)\rightarrow E(M)$
\bigskip
\[
\xymatrix{
M \ar[d]^{}  \ar[r]^{f} &M \ar[d]^{}\\
E(M) \ar[r]^{\varphi_f}                 &E(M).}
\]

\bigskip
 Now define $\eta: \End(M)\rightarrow \frac{S}{J(S)}$ by $\eta(f)=\varphi_f+J(S)$. It may be easily checked that $\eta$ is a ring homomorphism. Clearly, then $\eta \circ \psi': \End(M)\rightarrow \mathbb{F}_2$ is a ring homomorphism. This shows that $\End(M)$ has a factor isomorphic to $\mathbb{F}_2$, a contradiction to our hypothesis.    
Hence, $\End(E(M))$ has no factor isomorphic to $\mathbb{F}_2$. 
\end{proof}

\begin{lemma} $($\cite{KS}$)$ \label{2good} 
Let $M$ be a continuous right module over any ring $S$. Then each element of the endomorphism ring $R=\End(M_S)$ is the sum of two units if and only if $R$ has no factor isomorphic to $\mathbb{F}_2$. 
\end{lemma}

We can now prove our main result.

\begin{theorem} \label{char}
Let $M$ be any right $R$-module such that $\End(M)$ has no factor isomorphic to $\mathbb{F}_2$, then $M$ is quasi-injective if and only $M$ is automorphism-invariant.
\end{theorem}

\begin{proof}
Let $M$ be an automorphism-invariant right $R$-module such that $\End(M)$ has no factor isomorphic to $\mathbb{F}_2$. Then by Lemma \ref{basic}, $\End(E(M))$ has no factor isomorphic to $\mathbb{F}_2$. Now by Lemma \ref{2good}, each element of $\End(E(M))$ is a sum of two units. Therefore, for every endomorphism $\lambda \in \End(E(M))$, we have $\lambda=u_1+u_2$ where $u_1, u_2$ are automorphisms in $\End(E(M))$. As $M$ is an automorphism-invariant module, it is invariant under both $u_1$ and $u_2$, and we get that $M$ is invariant under $\lambda$. This shows that $M$ is quasi-injective. The converse is obvious.
\end{proof}

\begin{lemma} \label{z2}
Let $R$ be any ring and $S$, a subring of its center $Z(R)$. If $\mathbb{F}_2$ does not admit a structure of right $S$-module, then for any right $R$-module $M$, the endomorphism ring $\End(M)$ has no factor isomorphic to $\mathbb{F}_2$.
\end{lemma}

\begin{proof}
Assume to the contrary that there is a ring homomorphism $\psi: \End_R(M) \rightarrow \mathbb{F}_2$. Now, define a map $\varphi: S\rightarrow \End_R(M)$ by the rule $\varphi(r)=\varphi_r$, for each $r\in S$, where $\varphi_r:M\rightarrow M$ is given as $\varphi_r(m)=mr$. Clearly $\varphi$ is a ring homomorphism since $S\subseteq Z(R)$ and so, the composition $\varphi \circ f$ gives a nonzero ring homomorphism from $S$ to $\mathbb{F}_2$, yielding a contradiction to the assumption that $\mathbb{F}_2$ does not admit a structure of right $S$-module.
\end{proof}

We can now extend the above mentioned result of Dickson and Fuller.

\begin{theorem} \label{main}
Let $A$ be an algebra over a field $\mathbb F$ with more than two elements. Then any right $A$-module $M$ is automorphism-invariant if and only if $M$ is quasi-injective.
\end{theorem}

\begin{proof}
Let $M$ be an automorphism-invariant right $A$-module. Since $A$ is an algebra over a field $\mathbb F$ with more than two elements, by Lemma \ref{z2}, it follows that $\mathbb{F}_2$ does not admit a structure of right $Z(A)$-module and therefore $\End(M)$ has no factor isomorphic to $\mathbb{F}_2$. Now, by Theorem \ref{char}, $M$ must be quasi-injective. The converse is obvious.
\end{proof}

As a consequence we have the following

\begin{corollary}
Let $R$ be any algebra over a field $\mathbb F$ with more than two elements. Then $R$ is of right invariant module type if and only if every indecomposable right $R$-module is automorphism-invariant. 
\end{corollary}

\begin{corollary}
If $A$ is an algebra over a field $\mathbb F$ with more than two elements such that $A$ is automorphism-invariant as a right $A$-module, then $A$ is right self-injective.
\end{corollary}

It is well-known that a group ring $R[G]$ is right self-injective if and only if $R$ is right self-injective and $G$ is finite (see \cite{C}, \cite{R}). Thus, in particular, we have the following

\begin{corollary}
Let $K[G]$ be automorphism-invariant, where $K$ is a field with more than two elements. Then $G$ must be finite.
\end{corollary}

\bigskip

\bigskip

\end{document}